\documentclass[12pt]{article}
\usepackage{graphicx,amsmath,latexsym,amssymb,amsthm}
\usepackage{graphicx}
\usepackage[table]{xcolor}
\usepackage{amssymb}
\usepackage[utf8]{inputenc}
\usepackage{amsmath}
\usepackage{float}
\usepackage{cite}
\usepackage[pagewise]{lineno}
\usepackage[hidelinks]{hyperref}

\restylefloat{table}
at 9pt  

\newtheorem{theorem}{Theorem}[section]

\newtheorem{definition}[theorem]{Definition}
\newtheorem{corollary}[theorem]{Corollary}
\newtheorem{lemma}[theorem]{Lemma}
\newtheorem{example}[theorem]{Example}
\newtheorem{remark}[theorem]{Remark}

\numberwithin{equation}{section}

\usepackage{cite}
\makeatletter \def\@cite#1#2{\textup{[{#1\if@tempswa , #2\fi}]}} \makeatother

\usepackage{enumitem}  
\usepackage{calc}
\setlist{labelindent=1pt,itemsep=.5em}
\setlist[itemize]{leftmargin=1.2cm}
\setlist[enumerate]{itemindent=0em,leftmargin=1.2cm}
\setlist[enumerate,1]{label={\upshape(\roman*)}}

\makeatletter
\newcommand{\subjclass}[2][2010]{%
	\let\@oldtitle\@title%
	\gdef\@title{\@oldtitle\footnotetext{#1 \emph{Mathematics subject classification}: #2}}%
}
\newcommand{\keywords}[1]{%
	\let\@@oldtitle\@title%
	\gdef\@title{\@@oldtitle\footnotetext{\emph{Keywords}: #1.}}%
}
\makeatother

\title{Perov type T-contractive Mappings on Cone b-Metric Spaces with
	Generalized c-Distance}%

\author{Talat Nazir$^1$, Mujahid Abbas$^{2,3}$, Sergei Silvestrov$^4$ \\
	\footnotesize $^1$Department of Mathematical Sciences, University of South Africa, \\
	\footnotesize Florida 0003, South Africa \\
	\footnotesize \text{dr.talatnazir@gmail.com}   \\
	\footnotesize $^2$Department of Mathematics, Government College University, \\
	\footnotesize Katchery Road, Lahore 54000, Pakistan \\
	\footnotesize $^3$Department of Mathematics and Applied
	Mathematics, University of Pretoria, \\
	\footnotesize Hatfield 002, Pretoria, South Africa \\
	\footnotesize \text{abbas.mujahid@gmail.com}  \\
	\footnotesize $^{4}$Department of Mathematics and Physics, School of Education, Culture and
	Communication, \\
	\footnotesize M{\"a}lardalen University, Box 883, 72123 V{\"a}ster{\aa}s, Sweden \\
	\footnotesize \text{sergei.silvestrov@mdh.se}}


\date{\today}

\begin{document}
	
	\maketitle
	
	\begin{abstract}
		Fixed point results of Perov type mapping which satisfy generalized
		$T$-contractive conditions in the setup of cone $b$-metric spaces associated
		with generalized $c$-distance are proved and illustrated by nontrivial examples.
	\end{abstract}
	
	\section{Introduction}
	
	Fixed point theory provides certain simple tools to solve various
	problems arising in nonlinear analysis. The metric fixed point theory deals
	with the structure of sets equipped with some notion of distance functions
	and operators satisfying certain contraction conditions.
	
	The concept of a $b$-metric space was described by Bakhtin \cite%
	{BAK} and Boriceanu \cite{bor} as an extension of a usual metric space. Some
	useful fixed point results in the framework of a $b$-metric space were
	established in \cite{hussain1, BMV, CZE, KHH}.
	
	Huang and Zhang \cite{HZ07} extended the metric space structure
	and defined a cone metric space by replacing the range of the distance
	function with a normed space equipped with an order induced by the cone.
	Then, they obtained some fixed point results in a cone metric space (also,
	see \cite{ZA1}). Later on, the cone metric space was extended to a cone $b$%
	-metric space in \cite{CSDS,NHMS}.
	
	The notion of $w$-distance on metric spaces was introduced by Kada
	et al.\  \cite{Kada96} for solving non-convex minimization based problems.
	While, the notion of $c$-distance on cone metric spaces was defined by Cho
	et al. \cite{CSW} that could be viewed as the cone model of a $w$-distance.
	The concept of a $wt$-distance was initiated by Hussain et al. \cite{HSA} on
	a $b$-metric space and several fixed point results were proved using the $wt$%
	-distance on a $b$-metric space (also, see \cite{FAR17, MCK17}). Recently,
	the notion of a generalized $c$-distance in the setup of cone $b$-metric
	space was described by Bao et al.\  \cite{Bao} and some fixed point theorems
	were proved in this new framework.
	
	On the other hand, Perov \cite{Perov64} employed contraction
	matrices rather than the contraction constants to extend well known Banach
	contraction principle. Fixed point results for Perov type contractions in a
	cone metric space were presented in \cite{Cvetkovic15, Cvetkovic17}. Abbas
	et al. \cite{ANR17} obtained common fixed points of multivalued Perov type
	contractive mappings the setup of a cone metric space equipped a directed
	graph.
	
	In this paper, we prove several fixed point results of Perov type
	generalized $T$-contractive mappings defined on cone $b$-metric space
	equipped with a generalized $c$-distance. The obtained results unify, extend and
	generalize several results in the current literature.
	
	\section{Preliminaries}
	In the sequel
	$\mathbb{N}, \mathbb{N}_{0}, \mathbb{R}_{+}, \mathbb{R}$
	will denote the set of natural numbers, natural numbers with $0$, positive real numbers and real numbers, respectively.
	
	Let $E$ be a real Banach space. A cone $P$ is the subset of $E$
	that satisfies
	\begin{enumerate}[label=\upshape{(\roman*)},leftmargin=30pt]
		\item $P$ is closed, non-empty and $P\neq \{ \theta \}$
		(where $\theta $ is the zero element of $E$);
		\item $\alpha ,\beta \in \mathbb{R}_{+},$ $\alpha ,\beta \geq 0$
		with $a,b\in P$\ implies $\alpha a+\beta b\in P;$
		\item $-P\cap P=\{ \theta \}$.
	\end{enumerate}
	
	With a cone $P$, the partial ordering $\preceq $\ on $E$ is
	defined as: $u_{1}\preceq u_{2}$ if $u_{2}-u_{1}\in P$ and conversely.
	Here $u_{1}\prec u_{2}$ is used to indicate that $u_{1}\preceq
	u_{2}$ but $u_{1}\neq u_{2}$ and $u_{1}\ll u_{2}$ means that $u_{2}-u_{1}$
	belongs to the interior of $P$ denoted by $intP.$ If $intP\neq \theta ,$
	then $P$ is known as solid cone.
	
	A cone $P$ is said to be semi monotone or normal if there exists a
	positive constant $\kappa >0$ such that $\theta \preceq u\preceq v$ implies
	that $\left \Vert u\right \Vert \leq \kappa \left \Vert v\right \Vert $\ for all
	$u,v\in P$. The smallest value of $\kappa $ is known as the normal constant
	of normal cone $P$. The definition of normal cone $P$ is also equivalent to the following condition that
	\begin{equation*}
		\inf \{ \parallel u_{1}+u_{2}\parallel :u_{1},u_{2}\in P\text{ and}\parallel
		u_{1}\parallel =\parallel u_{2}\parallel =1\}>0.
	\end{equation*}
	If $u=(u_{1},\dots,u_{n})^{T}$ and
	$v=(v_{1},\dots,v_{n})^{T}\in \mathbb{R}^{n}$,
	then $u\preceq v$ means that $u_{i}\leq v_{i},$ for all $i=1,\dots,n$.
	In this way, $P$ becomes a normal cone with $\kappa =1$ defined as
	$
	P=\{u=(u_{1},\dots,u_{n})^{T}\in \mathbb{R}^{n}:u_{i}\geq 0$ for $i=1,2,\dots,n\}
	$.
	
	\begin{definition}[\cite{CSDS, NHMS}] Let $X$ be a nonempty
		set and $E$ a real Banach space equipped with the partial ordering induced
		by the cone $P\subset E$. For a real number $b\geq 1$ and for all $%
		u_{1},u_{2},u_{3}\in X,$ suppose that the mapping $d:X\times X\rightarrow E$
		satisfies the following conditions:
		\begin{enumerate}[label=\upshape{(b\arabic*)}, leftmargin=*]
			\item \label{b1:defbmetric}
			$\theta \preceq d(u_{1},u_{2})$ and $d(u_{1},u_{2})=\theta $ if
			and only if $u_{1}=u_{2}$;
			
			\item \label{b2:defbmetric}
			$d(u_{1},u_{2})=d(u_{2},u_{1})$;
			
			\item \label{b3:defbmetric}
			$d(u_{1},u_{2})\preceq b[d(u_{1},u_{2})+d(u_{2},u_{3})].$
		\end{enumerate}
		Then $d$ is called a cone $b$-metric on $X$ and the pair $(X,d)$
		is known as a cone $b$-metric space.
		
		Clearly, for $b=1$, the cone $b$-metric space becomes a cone
		metric space. Moreover, for $E=\mathbb{R}$ and $P=[0,\infty )$, the cone $b$-metric on $X$ becomes $b$-metric on $X$.
	\end{definition}
	
	\begin{example}\label{ex1.2}
		Let $X=\left[ 0,1\right] ,$ $E=X\times
		X\ $and $P=\{ \left( x,y\right) \in E:x,y\geq 0\}=E.$ Define $d:X\times
		X\rightarrow E$ by $d(x,y)=\left( \left( \phi \left( x,y\right) \right)
		^{p},\alpha \left( \phi \left( x,y\right) \right) ^{p}\right) $, where $%
		\alpha \geq 1$ and $p\geq 1\ $and $\phi :X\times X\rightarrow
		\mathbb{R}_{+}$ satisfies $\phi \left( x,y\right) =0$ if and only if $x=y;$ $\phi
		\left( x,y\right) =\phi \left( y,x\right) \ $and $\phi \left( x,z\right)
		\leq \phi \left( x,y\right) +\phi \left( y,z\right) .$ Then $(X,d)$ is a
		complete cone $b$-metric space with $b\geq 2^{p-1}\geq 1$ along with $P$ a
		solid cone.
		
		Note that $d$ trivially satisfies \ref{b1:defbmetric} and \ref{b2:defbmetric}.
		The condition \ref{b3:defbmetric} is satisfied since $\left( a+b\right) ^{p}\leq 2^{p-1}\left( a^{p}+b^{p}\right)$ for any $a,b\in \mathbb{R}_{+},$ and hence
		\begin{eqnarray*}
			d(x,y) &=&\left( \left( \phi \left( x,y\right) \right) ^{p},\alpha \left(
			\phi \left( x,y\right) \right) ^{p}\right)  \\
			&\leq &\left( \left( \phi \left( x,z\right) +\phi \left( z,y\right) \right)
			^{p},\alpha \left( \phi \left( x,z\right) +\phi \left( z,y\right) \right)
			^{p}\right)  \\
			&\leq &2^{p-1}\left( \left( \phi \left( x,z\right) \right) ^{p}+\left( \phi
			\left( z,y\right) \right) ^{p},\alpha \left( \left( \phi \left( x,z\right)
			\right) ^{p}+(\phi \left( z,y\right) \right) ^{p}\right)  \\
			&=&2^{p-1}\left( \left( \phi (x,z\right) )^{p},\alpha \left( \phi \left(
			x,z\right) \right) ^{p})+(\left( \phi \left( z,y\right) \right) ^{p},\left(
			\phi \left( z,y\right) \right) ^{p}\right)  \\
			&\leq &b\left( d\left( x,z\right) +d\left( z,y\right) \right).
		\end{eqnarray*}
		Thus $d$ is a cone $b$-metric on $X$.
	\end{example}
	
	\begin{definition} Let $X$ be a $b$-cone metric space, $a\in E$ where $\theta \ll a$, and $\{u_{n}\}$ a sequence in $X.$ Then
		$\{u_{n}\}$ is called
		\begin{enumerate}[label=\upshape{(\roman*)},leftmargin=30pt]
			\item Cauchy sequence if there exists an $n_{0}$ in $
			\mathbb{N}
			$ such that $d(u_{n},u_{m})\ll a$ for all $n,m\geq n_{0}$;
			\item convergent if there exists an $n_{0}$ in
			$\mathbb{N}$ and $u\in X$ such that $d(u_{n},u)\ll a$ for all $n\geq n_{0}$;
			\item If every Cauchy sequence in $X$ is convergent, then $X$ is
			called a complete cone metric space.
		\end{enumerate}
		A cone $b$-metric space $X$ is said to be complete if every Cauchy
		sequence $\{u_{n}\}$ in $X$ is convergent in $X$.
		
		If a cone is normal, then the sequence $\{u_{n}\}$ converges to a
		point $u\in X$ if and only if $d(u_{n},u)\rightarrow 0$ as $n\rightarrow
		\infty. $
	\end{definition}
	\begin{definition}[\cite{Bao}]
		\label{de1.4}
		Let $(X,d)$ be a cone $b$-metric space. For $b\geq 1$, the mapping $q:X\times X\rightarrow E$ is
		called a generalized $c$-distance on $X$ if for any $u_{1},u_{2},u_{3}\in X,$
		the following properties are satisfied:
		
		\begin{enumerate}[label=\upshape{(q$_\arabic*$)}, leftmargin=*]%
			\item \label{q1:defgencdist}  
			$\theta \preceq q(u_{1},u_{2})$;
			\item \label{q2:defgencdist} 
			$q(u_{1},u_{3})\preceq b[q(u_{1},u_{2})+q(u_{2},u_{3})]$;
			\item \label{q3:defgencdist} 
			If $q(u,v_{n})\preceq e$ for all $n\geq 1$, where $e\in P,$
			then $q(u,v)\preceq be$, where $\{v_{n}\} \subseteq $ $X$ be a sequence
			convergent to $v\in X$;
			\item \label{q4:defgencdist} 
			If $e^{\ast }\in intP$, then an $e\in E$ with $\theta \ll e$
			exists such that for $q(u_{3},u_{1})\ll e$ and $q(u_{3},u_{2})\ll e$ gives $%
			d(u_{1},u_{2})\ll e^{\ast }$.
		\end{enumerate}
		
		For any $b$-metric space $(X,d)$, where $E=\mathbb{R}$ and $P=[0,\infty ),$ the $wt$-distance on a $b$-metric space $X$ is the generalized $c$-distance, but the converse does not hold in general \cite{HSA}.
		
		Also for $b=1$, the generalized $c$-distance becomes $c$-distance
		\cite{CSW}. Note that, for $b=1,$ $E=\mathbb{R}$ and $P=[0,\infty )$ it becomes
		the $w$-distance given in \cite{Kada96}. Moreover, $q(u_{1},u_{2})=\theta $
		does not imply $u_{1}=u_{2}$ and $q(u_{1},u_{2})\neq q(u_{2},u_{1})$ for
		all $u_{1},u_{2}\in X$.
	\end{definition}
	
	\begin{example}
		\label{ex1.5}
		Let $X=[0,1],$
		$E=\mathbb{R}^{2}\ $ and $P=\{ \left( x,y\right) \in E:x,y\geq 0\} \subseteq E.$
		Define $d:X\times X\rightarrow E$ as given in Example \ref{ex1.2}. Then $(X,d)$ is a
		complete cone $b$-metric space with $b\geq 2^{p-1}\geq 1$ along with $P$ a
		solid cone.
		
		Define $q\left( x,y\right) =\left( \psi \left( y\right) ,\eta \psi
		\left( y\right) \right) ,$ where $\eta \geq 1$ and $\psi :X\rightarrow
		\mathbb{R}_{+}$ is of the form $\psi \left( t\right) =at^{p}$ for some $p>1,$ where
		$a>0.$ Then $q$ is a generalized\ $c$-distance on $X.$
		
		Note that $q$ is trivially satisfying \ref{q1:defgencdist} since for all $x,y\in X,$
		$$q\left(x,y\right) =\left( \psi \left( y\right) ,\eta \psi \left( y\right) \right)
		=\left( ay^{p},\eta ay^{p}\right) \succeq \theta. $$
		
		For \ref{q2:defgencdist}, for any $x,y,z\in X,$
		\begin{eqnarray*}
			q\left( x,z\right)  &=&\left( \psi \left( z\right) ,\eta \psi \left(
			z\right) \right)  \\
			&=&\left( az^{p},\eta az^{p}\right)  \\
			&\leq &b\left( a\left( y^{p}+z^{p}\right) ,\eta a\left( y^{p}+z^{p}\right)
			\right)  \\
			&=&b\left( \left( ay^{p},\eta ay^{p}\right) +\left( az^{p},\eta
			az^{p}\right) \right)  \\
			&=&b\left( q\left( x,y\right) +q\left( y,z\right) \right) .
		\end{eqnarray*}%
		For \ref{q3:defgencdist}, one needs to show that for any $\left \{ v_{n}\right \} $ in $X$
		that converges to $v$ in $X,$ if $q(u,v_{n})\preceq e$ for all $n\geq 1$,
		where $e=\left( e_{1},e_{2}\right) \in P,$ then $q\left( u,v\right) \preceq be.$
		Since $q\left( u,v_{n}\right) =\left( \psi \left( v_{n}\right) ,\eta
		\psi \left( v_{n}\right) \right) =\left( av_{n}^{p},\eta av_{n}^{p}\right)
		\preceq e,$ that is $\left( e_{1}-av_{n}^{p},e_{2}-\eta av_{n}^{p}\right)
		\in P$ or equivalently, $e_{1}-av_{n}^{p}\geq 0\ $and $e_{2}-\eta
		av_{n}^{p}\geq 0,$ taking limit as $n\rightarrow \infty $ yields $%
		e_{1}-av^{p}\geq 0\ $and $e_{2}-\eta av^{p}\geq 0,$
		which gives $q\left( u,v\right) =\left( \psi \left( v\right) ,\eta
		\psi \left( v\right) \right) =\left( av^{p},\eta av^{p}\right) \preceq e\leq
		be.$
		
		Finally to show \ref{q4:defgencdist}, we need to show that if $e^{\ast }\in
		intP$, then an $e\in E$ with $\theta \ll e$ exists such that for $%
		q(u_{3},u_{1})\ll e$ and $q(u_{3},u_{2})\ll e$ gives $d(u_{1},u_{2})\ll
		e^{\ast }.$
		
		Condition $q(u_{3},u_{1})=\left( \psi \left( u_{1}\right) ,\eta
		\psi \left( u_{1}\right) \right) =\left( au_{1}^{p},\eta au_{1}^{p}\right)
		\ll e=\left( e_{1},e_{2}\right) $ is equivalent to
		$e_{1}-\psi \left( u_{1}\right) \geq 0$ and $e_{2}-\eta \psi
		\left( u_{1}\right) \geq 0$, and the condition
		$q(u_{3},u_{2})=\left( \psi \left( u_{2}\right) ,\eta \psi \left(
		u_{2}\right) \right) =\left( au_{2}^{p},\eta au_{2}^{p}\right) \ll e\ $
		is equivalent to
		$e_{1}-\psi \left( u_{2}\right) \geq 0$ and $e_{2}-\eta \psi
		\left( u_{2}\right) \geq 0.$
		
		Now
		\begin{eqnarray*}
			d(u_{1},u_{2}) &=&\left( \left( \phi \left( u_{1},u_{2}\right) \right)
			^{p},\alpha \left( \phi \left( u_{1},u_{2}\right) \right) ^{p}\right)  \\
			&\leq &\left( \left( \phi \left( u_{1},u_{3}\right) +\phi \left(
			u_{3},u_{2}\right) \right) ^{p},\alpha \left( \phi \left( u_{1},u_{3}\right)
			+\phi \left( u_{3},u_{2}\right) \right) ^{p}\right)  \\
			&=&2^{p-1}\left( \left( \phi \left( u_{1},u_{3}\right) )^{p}+(\phi \left(
			u_{3},u_{2}\right) )^{p},\alpha \left( \left( \phi \left( u_{1},u_{3}\right)
			\right) ^{p}+(\phi \left( u_{3},u_{2}\right) \right) ^{p}\right) \right)  \\
			&=&2^{p-1}\left( \left( (\phi \left( u_{1},u_{3}\right) )^{p},\alpha \left(
			\left( \phi \left( u_{1},u_{3}\right) \right) ^{p})+((\phi \left(
			u_{3},u_{2}\right) )^{p},\alpha (\phi \left( u_{3},u_{2}\right) \right)
			^{p}\right) \right)  \\
			&\leq &2^{p-1}(\left( k_{1}\psi \left( u_{1}\right) ,\alpha k_{1}\psi \left(
			u_{1}\right) \right) +\left( k_{2}\psi \left( u_{2}\right) ,\alpha k_{2}\psi
			\left( u_{2}\right) \right) ) \\
			&=&2^{p-1}(k_{1}\left( e_{1},\frac{\alpha }{\eta }e_{2}\right) +k_{2}\left(
			e_{1},\frac{\alpha }{\eta }e_{2}\right) ) \quad \quad \mbox{
				\textup{(}$k_{1},k_{2}\geq 0$\textup{)}} \\
			&=&2^{p-1}\left( k_{1}+k_{2}\right) \left( e_{1},\frac{\alpha }{\eta }%
			e_{2}\right)  \\
			&=&\left( e_{1}^{\ast },e_{2}^{\ast }\right)  \\
			&=&e^{\ast },
		\end{eqnarray*}
		where $e=\left( e_{1},e_{2}\right) =\left( \frac{1}{2^{p-1}\left(
			k_{1}+k_{2}\right) }e_{1}^{\ast },\frac{\eta }{2^{p-1}\alpha \left(
			k_{1}+k_{2}\right) }e_{2}^{\ast }\right) .$
		Clearly $\theta \ll e,$ and for any $u_{1}$ and $u_{2}$ such that
		$q(u_{3},u_{1})=\left( au_{1}^{p},\eta au_{1}^{p}\right) \ll e$ and $q(u_{3},u_{2})=\left( au_{2}^{p},\eta au_{2}^{p}\right) \ll e$ holds $d(u_{1},u_{2})\ll e^{\ast }.$
		
		Note that the conditions $q(u_{3},u_{1})=\left( au_{1}^{p},\eta
		au_{1}^{p}\right) \ll e$ and $q(u_{3},u_{2})=\left( au_{2}^{p},\eta
		au_{2}^{p}\right) \ll e$ are equivalent to
		\begin{align*}
			u_{1} & \leq \frac{1}{\sqrt[p]{2^{p-1}a\left( k_{1}+k_{2}\right) }}\sqrt[p]{%
				e_{1}^{\ast }} \\
			u_{1} & \leq \frac{1}{\sqrt[p]{2^{p-1}a\alpha \left(
					k_{1}+k_{2}\right) }}\sqrt[p]{e_{2}^{\ast }}, \\
			u_{2} & \leq \frac{\eta }{\sqrt[p]{2^{p-1}a\left(
					k_{1}+k_{2}\right) }}\sqrt[p]{e_{1}^{\ast }}, \\
			u_{2} & \leq \frac{\eta }{\sqrt[p]{2^{p-1}a\alpha \left( k_{1}+k_{2}\right) }}\sqrt[p]{e_{2}^{\ast }},
		\end{align*}
		which can also be written as
		\begin{multline*}
			\left( u_{1},u_{2}\right) \leq \left( \min \{ \frac{\sqrt[p]{
					e_{1}^{\ast }}}{\sqrt[p]{2^{p-1}a\left( k_{1}+k_{2}\right) }},\frac{\sqrt[p]{
					e_{2}^{\ast }}}{\sqrt[p]{2^{p-1}a\alpha \left( k_{1}+k_{2}\right) }}\},\right. \\
			\left. \min \{ \frac{\eta \sqrt[p]{e_{1}^{\ast }}}{\sqrt[p]{2^{p-1}a\left(
					k_{1}+k_{2}\right) }},\frac{\eta \sqrt[p]{e_{2}^{\ast }}}{\sqrt[p]{
					2^{p-1}a\alpha \left( k_{1}+k_{2}\right) }}\} \right) .
		\end{multline*}
	\end{example}
	
	The following useful lemma is needed in the sequel.
	\begin{lemma}\label{lemma:cdistconv}
		Let $q$ be a generalized $c$-distance on
		$X $ equipped with cone $b$-metric space $(X,d).$ For sequences $\{x_{n}\}$
		and $\{y_{n}\}$ in $X$, two
		sequences $\{a_{n}\}$ and $\{b_{n}\}$  in $P$ converging to $\theta $ and for any $u_{1},u_{2},u_{3}\in X$, the following statements hold:
		
		\begin{enumerate}[label=\upshape{(\roman*)},leftmargin=30pt]
			\item \label{ilemma:cdistconv} if for all $n\in \mathbb{N}$, $q(x_{n},u_{1})\preceq a_{n}$ and $%
			q(x_{n},u_{2})\preceq b_{n}$, then $u_{1}=u_{2}$. In particular, for $%
			q(u_{3},u_{1})=\theta $ and $q(u_{3},u_{2})=\theta $ implies $u_{1}=u_{2}$;
			\item \label{iilemma:cdistconv} if for all $n\in \mathbb{N}$, $q(x_{n},y_{n})\preceq a_{n}$ and
			$q(x_{n},u_{1})\preceq b_{n}$, then $\{y_{n}\}$ converges to $u_{1}$;
			\item \label{iiilemma:cdistconv} if $q(x_{n},x_{m})\preceq a_{n}$ for all $m,n\in \mathbb{N}$
			with $m\geq n,$ then $\{x_{n}\}$ is a Cauchy sequence in $X$;
			\item \label{ivlemma:cdistconv} if $q(u_{1},x_{n})\preceq a_{n}$ for all $n\in \mathbb{N}$,
			then $\{x_{n}\}$ is a Cauchy sequence in $X$.
		\end{enumerate}
	\end{lemma}
	
	Let $M_{n\times n\text{ }}(\mathbb{R}^{+})$ denote the set of all $n\times n$ matrices with non negative
	elements. It is well known that if $A$ is any square matrix of order $n$,
	then $A(P)\subset P$ if and only if $A\in M_{n,n}(\mathbb{R}^{+})$.
	A matrix $A\in M_{n,n}(\mathbb{R}^{+})$ is called convergent to zero if
	$A^{n}\longrightarrow \Theta $ as $n\longrightarrow \infty $, where $\Theta $ is the null matrix of size $n.$
	
	Regarding this class of matrices, we have the following classical result in matrix analysis
	(see \cite{Precup09}, \cite{Rus79} and \cite{Turinici90}).
	
	\begin{theorem}\label{1.7}
		Let $A\in M_{n,n}(\mathbb{R}^{+})$. The following statements are equivalent:
		\begin{enumerate}[label=\upshape{(\roman*)},leftmargin=30pt]
			\item $A^{n}\rightarrow \Theta $, as $n\rightarrow \infty $;
			\item eigenvalues of $A$ lie in the open unit disc, that is, for all $%
			\lambda \in C$ satisfying $\det (A-\lambda I_{n})=0,$ we have $\left \vert
			\lambda \right \vert <1$ ;
			\item matrix $I_{n}-A$ is non-singular and $%
			(I_{n}-A)^{-1}=I_{n}+A+A^{2}+...+A^{m}+...$;
			\item matrix $(I_{n}-A)$ is non-singular and $(I_{n}-A)^{-1}$ has
			non-negative elements.
			
		\end{enumerate}
	\end{theorem}
	
	Perov \cite{Perov64}$\ $obtained the following generalization of
	Banach contraction principle for the case when metric $d$ take values in
	$\mathbb{R}_{+}^{n},$ that is, $d\left( x,y\right) \in
	\mathbb{R}_{+}^{n}$ and known as generalized metric space.
	
	\begin{theorem}\label{1.8}
		Let $(X,d)$ be a complete generalized
		metric space, $f$ be a self mapping on $X$ and $A\in M_{n,n}(\mathbb{R}^{+})$ a matrix convergent to zero. If for any $u,v\in X,$ we have
		\begin{equation*}
			d(f(u),f(v))\leq A(d(u,v)),
		\end{equation*}
		then the following statements hold:
		\begin{enumerate}[label=\upshape{\arabic*.},leftmargin=30pt]
			\item  $f$ has a unique fixed point $u^{\ast }\in X.$
			\item  The Picard iterative sequence $u_{n}=f^{n}(u_{0}),$ $n\in\mathbb{N}$
			converges to $u^{\ast }$\ for all $u_{0}\in X$.
			\item  $d(u_{n},u^{\ast })\leq A^{n}(I_{n}-A)^{-1}(d(u_{0},u_{1})),$
			$n\in\mathbb{N}.$
			\item  if $g:X\rightarrow X$ satisfies the condition $d(f(u),g(u))\leq a$
			for all $u\in X$ and some $a\in\mathbb{R}^{n},$ then for the sequence
			$v_{n}=g^{n}(u_{0}),$ $n\in\mathbb{N},$ the inequality
			\begin{equation*}
				d(v_{n},u^{\ast })\leq (I_{n}-A)^{-1}(a)+A^{n}(I_{n}-A)^{-1}(d(u_{0},u_{1}))
			\end{equation*}
			is valid for all $n\in\mathbb{N}.$
		\end{enumerate}
	\end{theorem}
	
	The role of vector valued norm has much importance in the study of
	semi-linear operator systems. For details, we refer to \cite{Precup09}.
	Let $\mathcal{B}(E)$ be the set of all bounded linear operators on
	$E$, and $L(E)$ be the set of all linear operators on $E$.
	
	Note that $\mathcal{B}(E)$ is a Banach algebra. If $A\in
	\mathcal{B}(E)$, then
	\begin{equation*}
		r(A)=\lim \limits_{n\rightarrow \infty }\left \Vert A^{n}\right \Vert ^{\frac{1}{n}}=\inf_{n}\left \Vert A^{n}\right \Vert ^{\frac{1}{n}}
	\end{equation*}
	is the spectral radius of $A$. We\ write $\mathcal{B}(E)^{-1}$ for the set
	of all invertible elements in $\mathcal{B}(E).$
	Let us remark that if $r(A)<1 $, then
	\begin{enumerate}[label=\upshape{(\roman*)},leftmargin=30pt]
		\item series $\sum \limits_{n=0}^{\infty }A^{n}$ is absolutely
		convergent;
		\item the element $I-A$ is invertible in $\mathcal{B}(E)$ and
		$\sum \limits_{n=0}^{\infty }A^{n}=(I-A)^{-1}.$
	\end{enumerate}
	If $A_{1},A_{2}\in \mathcal{B}(E)$ and $A_{1}A_{2}=A_{2}A_{1}$
	then $r(A_{1}A_{2})\leq r(A_{2})r(A_{1}).$
	If\ $A\in \mathcal{B}(E)$ and $A^{-1}\in \mathcal{B}(E)$ exists,
	then $r(A^{-1})=1/r(A).$
	Further, if $\parallel A\parallel <1$, then $I-A$ is invertible and%
	\begin{equation*}
		\left \Vert (I-A)^{-1}\right \Vert \leq \frac{1}{1-\left \Vert A\right \Vert}.
	\end{equation*}
	Note that $r(A)\leq \left \Vert A\right \Vert $.
	
	\begin{remark}[\cite{Cvetkovic15}] \label{1.9}
		Let $X$ be a $b$-cone
		metric space, $P\subseteq E$ be a cone in $E$ and $A:E\rightarrow E$ be a linear
		operator. The following conditions are equivalent:
		\begin{enumerate}[label=\upshape{a\arabic*)}, leftmargin=*]
			\item $A$ is increasing, that is$,$ $u\preceq v$ implies that $%
			A(u)\preceq A(v);$
			\item $A$ is positive, that is, $A(P)\subset P;$
			\item If $\left \Vert A\right \Vert <1$ and $A$ is positive, then
			inverse of $I-A$ is positive.
		\end{enumerate}
	\end{remark}
	\begin{remark}[\cite{ANR17}] \label{1.10}
		Let $P\subseteq E$ be a cone in $E.$
		Suppose that $A:E\rightarrow E$ is a bounded linear operator with
		$A\left( P\right) \subset P$ and $r(A)<1$. If $u\preceq A(u)$ for $u$ in $P$,
		then $u=\theta $.
	\end{remark}
	
	Now we give the definitions of following Perov type $T$-Hardy
	Rogers contractions in the setup of $b$-cone metric spaces via generalized $c$-distance.
	\begin{definition}{1.11}
		Let $(X,d)$ be a $b$-cone metric
		space, $q$ a $c$-distance on $X.$ Suppose that $f,T:X\rightarrow X$ are two
		self mappings. A mapping $f$ is said to be:
		
		\begin{enumerate}
			\item[(I)] a Perov $THR_{1}-$contraction if there exists\ a linear bounded
			operator $A:E\rightarrow E$ with $r(bA)<1$ and $A\left( P\right) \subset P$
			such that for all $x_{1},x_{2}\in X$,%
			\begin{equation}
				q(Tfx_{1},Tfx_{2})\leq A(U_{T}(x_{1},x_{2})),  \tag{1.1}
			\end{equation}%
			holds, where%
			\begin{equation*}
				U_{T}(x_{1},x_{2})\in
				\{q(Tx_{1},Tx_{2}),q(Tx_{1},Tfx_{1}),q(Tx_{2},Tfx_{2})\}.
			\end{equation*}
			
			\item[(II)] a Perov $THR_{2}-$contraction if there exist\ linear bounded
			operators $A_{1},A_{2},A_{3}:E\rightarrow E$ with $A_{k}\left( P\right)
			\subset P$ for $k=1,2,3$, $(I-A_{3})^{-1}\in \mathcal{B}(E)$ and $%
			r[b(I-A_{3})^{-1}(A_{1}+A_{2})]<1$ such that for all $x_{1},x_{2}\in X,$%
			\begin{equation}
				q(Tfx_{1},Tfx_{2})\leq A_{1}\left( q(Tx_{1},Tx_{2})\right) +A_{2}\left(
				q(Tx_{1},Tfx_{1})\right) +A_{3}\left( q(Tx_{2},Tfx_{2})\right) ,  \tag{1.2}
			\end{equation}
			holds.
		\end{enumerate}
	\end{definition}
	
	\begin{definition}\label{1.12}
		Let $(X,d)$ be a $b$-cone metric
		space and $P$ a solid cone. A mapping $T:X\rightarrow X$ is said to be
		\begin{enumerate}
			\item[a)] continuous if, for all $\{u_{n}\}$ in $X$ such that $\lim
			\limits_{n\rightarrow \infty }u_{n}=u^{\ast }$ implies that $\lim
			\limits_{n\rightarrow \infty }Tu_{n}=Tu^{\ast }.$
			\item[b)] sequentially convergent if, for every sequence $\{u_{n}\}$ in $X$
			such that $\{Tu_{n}\}$ is convergent implies that $\{u_{n}\}$ is also
			convergent.
		\end{enumerate}
	\end{definition}
	\section{Fixed Points Results}
	
	In this section, we obtain several fixed point results for $T$%
	-Hardy Rogers and Perov type contractive mappings via $c$-distances in the
	setup of cone $b$-metric spaces.
	
	We start with the following result.
	
	\begin{theorem}\label{2.1}
		Let $(X,d)$ be a complete cone $b$%
		-metric space, $P$ a solid cone and $q$ a generalized\ $c$-distance on $X.$
		Suppose that $T:X\rightarrow X$ is a one to one, continuous and sequentially
		convergent. If the mapping $f:X\rightarrow X$ is a Perov $THR_{1}-$%
		contraction, then $f$ has a unique fixed point $u^{\ast }\in X.$ And for any
		$x_{0}\in X,$ the iterative sequence $\{f^{n}x_{0}\}$ converges to the fixed
		point of $f$. Moreover $q(Tu^{\ast },Tu^{\ast })=\theta $ provided that $%
		u^{\ast }=fu^{\ast }$.
	\end{theorem}
	\begin{proof}
		Let $x_{0}$ be any element in $X.$ Set $%
		x_{1}=fx_{0},$ $x_{2}=fx_{1}=f^{2}x_{0},..._{\cdot }$ Define a sequence $%
		\{x_{n}\}$ in $X$ by $x_{n+1}=f^{n}x_{0}=fx_{n}.$ As the mapping $f$ is a
		Perov $THR_{1}-$contraction, we have
		\begin{equation}\label{2.11}
			q(Tx_{n},Tx_{n+1})=q(Tfx_{n-1},Tfx_{n}))\leq A(U_{T}(x_{n-1},x_{n})),
		\end{equation}
		where
		\begin{eqnarray*}
			U_{T}(x_{n-1},x_{n}) &\in
			&\{q(Tx_{n-1},Tx_{n}),q(Tx_{n-1},Tfx_{n-1}),q(Tx_{n},Tfx_{n})\} \\
			&=&\{q(Tx_{n-1},Tx_{n}),q(Tx_{n},Tx_{n+1})\}.
		\end{eqnarray*}%
		If $U_{T}(x_{n-1},x_{n})=q(Tx_{n},Tx_{n+1}),$ then
		$q(Tx_{n},Tx_{n+1})\leq A(q(Tx_{n},Tx_{n+1}))$,
		which by the Remark \ref{1.10}
		implies that $q(Tx_{n},Tx_{n+1})=\theta.$
		Thus, for all $n\in \mathbb{N},$
		\begin{equation*}
			q(Tx_{n},Tx_{n+1})\leq A(q(Tx_{n-1},Tx_{n})).
		\end{equation*}%
		Continuing this way, we get
		$
		q(Tx_{n},Tx_{n+1})\leq A^{n}(q(Tx_{0},Tx_{1})).
		$
		For any $m,n\in\mathbb{N}$ with $m>n,$ it follows from $(q_{2})$ that
		\begin{align}\label{2.22}
			q(Tx_{n},Tx_{m})& \leq b[q(Tx_{n},Tx_{n+1})+q(Tx_{n+1},Tx_{m})]  \notag \\
			& \leq bq(Tx_{n},Tx_{n+1})+b[bq(Tx_{n+1},Tx_{n+2})+q(Tx_{n+2},Tx_{m})]
			\notag \\
			& \leq \cdot \cdot \cdot   \notag \\
			& \leq bq(Tx_{n},Tx_{n+1})+b^{2}q(Tx_{n+1},Tx_{n+2})+\cdots
			+b^{m-n}q(Tx_{m-1},Tx_{m})  \notag \\
			& \leq (bA^{n}+b^{2}A^{n+1}\cdots +b^{m-n}A^{m-1})(q(Tx_{0},Tx_{1}))  \notag
			\\
			& \leq bA^{n}\left( I-bA\right) ^{-1}(q(Tx_{0},Tx_{1})).
		\end{align}%
		Let $c\gg 0.$ Choose $\delta >0$ such that $c+N_{\delta }(\theta )\subseteq
		P\,,$ where $N_{\delta }(\theta )=\{x\in E:\left \Vert x\right \Vert <\delta
		\}.$ Also, choose $N_{1}\in\mathbb{N}$
		such that $bA^{n}(I-bA)^{-1}(q(Tx_{0},Tx_{1}))\in N_{\delta }(\theta )$
		for all $n>N_{1}$. Thus for all $m>n>N_{1},$ we have
		\begin{equation*}
			q(Tx_{n},Tx_{m})\preceq bA^{n}(I-bA)^{-1}(d(Tx_{0},Tx_{1}))\ll c
		\end{equation*}
		which implies that $\{Tx_{n}\}$ is a Cauchy sequence in $X$. Since $X$ is
		complete, there exists an element point $v\in X$ such that
		$Tx_{n}\rightarrow v$ as $n\rightarrow \infty $.
		
		Since $T$ is subsequently convergent, $\{x_{n}\}$ has a convergent
		subsequence. So, there is $x^{\ast }\in X$ and subsequence $\{x_{n_{i}}\}$
		such that $x_{n_{i}}\rightarrow x^{\ast }$ as $i\rightarrow \infty $. As $T$
		is continuous, we obtain $\lim \limits_{n_{i}\rightarrow \infty
		}Tx_{n_{i}}=Tx^{\ast }$. The uniqueness of the limit implies that $Tx^{\ast
		}=v$. By (q$_{\text{3}}$), we obtain
		\begin{equation}\label{e2.3}
			q(Tx_{n},Tx^{\ast })\leq bA^{n}(I-bA)^{-1}(q(Tx_{0},Tx_{1})).
		\end{equation}%
		Also%
		\begin{equation}\label{e2.4}
			\begin{array}{cc}
				q(Tx_{n},Tfx^{\ast }) & =q(Tfx_{n-1},Tfx^{\ast }) \\
				& \leq A(U_{T}(x_{n-1},x^{\ast })),%
			\end{array}
		\end{equation}%
		where
		\begin{eqnarray*}
			U_{T}(x_{n-1},x^{\ast }) &\in &\{q(Tx_{n-1},Tx^{\ast
			}),q(Tx_{n-1},Tfx_{n-1}),q(Tx^{\ast },Tfx^{\ast })\} \\
			&=&\{q(Tx_{n-1},Tx^{\ast }),q(Tx_{n-1},Tx_{n}),q(Tx^{\ast },Tfx^{\ast })\}.
		\end{eqnarray*}%
		Now, if $U_{T}(x_{n-1},x^{\ast })=q(Tx_{n-1},Tx^{\ast }),$ then we obtain%
		\begin{eqnarray*}
			q(Tx_{n},Tfx^{\ast }) &\leq &A\left( q(Tx_{n-1},Tx^{\ast })\right)  \\
			&\leq &bA^{n-1}(I-bA)^{-1}(q(Tx_{0},Tx_{1})).
		\end{eqnarray*}%
		Whenever $U_{T}(x_{n-1},x^{\ast })=q(Tx_{n-1},Tx_{n}),$ then we have%
		\begin{eqnarray*}
			q(Tx_{n},Tfx^{\ast }) &\leq &A\left( q(Tx_{n-1},Tx_{n})\right)  \\
			&\leq &bA^{n}(I-bA)^{-1}(q(Tx_{0},Tx_{1})).
		\end{eqnarray*}%
		Finally if $U_{T}(x_{n-1},x^{\ast })=q(Tx^{\ast },Tfx^{\ast }),$ then we get
		that%
		\begin{eqnarray*}
			q(Tx_{n},Tfx^{\ast }) &\leq &A\left( q(Tx^{\ast },Tfx^{\ast })\right)  \\
			&\leq &bA\left( q(Tx^{\ast },Tx_{n})\right) +bA\left( q(Tx_{n},Tfx^{\ast
			})\right) ,
		\end{eqnarray*}%
		that is,%
		\begin{eqnarray*}
			q(Tx_{n},Tfx^{\ast }) &\leq &bA\left( I-bA\right) ^{-1}\left( q(Tx^{\ast
			},Tx_{n})\right)  \\
			&\leq &bA^{n}\left( I-bA\right) ^{-1}\left( q(Tx_{0},Tx_{1})\right) .
		\end{eqnarray*}
		Then by Lemma \ref{lemma:cdistconv} \ref{ilemma:cdistconv}, \eqref{e2.3} and \eqref{e2.4},
		we have $Tx^{\ast}=Tfx^{\ast }$. As the map $T$ is one to one, so $x^{\ast }=fx^{\ast },$
		that is, $x^{\ast }$ is the fixed point of $f.$
		
		Now by taking $x^{\ast }=fx^{\ast },$ we have
		\begin{equation}\label{2.5}
			q(Tx^{\ast },Tx^{\ast })  =q(Tfx^{\ast },Tfx^{\ast })
			\leq A(U_{T}(x^{\ast },x^{\ast })),
		\end{equation}
		where
		\begin{eqnarray*}
			U_{T}(x^{\ast },x^{\ast }) &\in &\{q(Tx^{\ast },Tx^{\ast }),q(Tx^{\ast
			},Tfx^{\ast }),q(Tx^{\ast },Tfx^{\ast })\} \\
			&=&\{q(Tx^{\ast },Tx^{\ast }),q(Tx^{\ast },Tx^{\ast }),q(Tx^{\ast },Tx^{\ast})\} \\
			&=&\{q(Tx^{\ast },Tx^{\ast })\}.
		\end{eqnarray*}
		Thus, $q(Tx^{\ast },Tx^{\ast })\leq A(q(Tx^{\ast },Tx^{\ast })).$
		By Remark \ref{1.10}, we get $q(Tx^{\ast },Tx^{\ast })=\theta .$
		
		To prove the uniqueness of fixed point, suppose that there exists
		another point $y^{\ast }$ in $X$ such that $y^{\ast }=fy^{\ast }$. Then we
		have
		\begin{equation}\label{2.6}
			q(Tx^{\ast },Ty^{\ast })  =q(Tfx^{\ast },Tfy^{\ast })
			\leq A(U_{T}(x^{\ast },y^{\ast })),
		\end{equation}
		where
		\begin{eqnarray*}
			U_{T}(x^{\ast },y^{\ast }) &\in &\{q(Tx^{\ast },Ty^{\ast }),q(Tx^{\ast
			},Tfx^{\ast }),q(Ty^{\ast },Tfy^{\ast })\} \\
			&=&\{q(Tx^{\ast },Ty^{\ast }),q(Tx^{\ast },Tx^{\ast }),q(Ty^{\ast },Ty^{\ast
			})\} \\
			&=&\{q(Tx^{\ast },Ty^{\ast }),\theta \}.
		\end{eqnarray*}%
		Now if $U_{T}(x^{\ast },y^{\ast })=q(Tx^{\ast },Ty^{\ast })$, then we have%
		\begin{equation*}
			q(Tx^{\ast },Ty^{\ast })\leq A(q(Tx^{\ast },Ty^{\ast })),
		\end{equation*}%
		which by the Remark \ref{1.10} implies $q(Tx^{\ast },Ty^{\ast })=\theta $.
		Thus, by Lemma \ref{lemma:cdistconv} \ref{ilemma:cdistconv}, we obtain $Tx^{\ast }=Ty^{\ast }.$
		
		In case $U_{T}(x^{\ast },y^{\ast })=\theta $, we have
		$q(Tx^{\ast },Ty^{\ast })\leq A(\theta )=\theta,$
		that is, $q(Tx^{\ast },Ty^{\ast })=\theta $. Again, by the Lemma \ref{lemma:cdistconv} \ref{ilemma:cdistconv}, it follows that $Tx^{\ast }=Ty^{\ast }.$
		Since $T$ is one to one, so $x^{\ast }=y^{\ast }$. Thus $f$ has a
		unique fixed point.
	\end{proof}
	The following non trivial example shows the validity of Theorem \ref{2.1}.
	
	\begin{example}
		Let $X=[0,1],$ $E=\mathbb{R}^{2}\ $and $P=\{ \left( x,y\right) \in E:x,y\geq 0\} \subseteq E.$
		Define $d:X\times X\rightarrow E$ as given in Example \ref{ex1.2}. Then $(X,d)$ is a
		complete cone $b$-metric space with $b\geq 2^{p-1}\geq 1$ along with $P$ a
		solid cone.
		
		Define $q\left( x,y\right) =\left( \psi \left( y\right) ,\eta \psi
		\left( y\right) \right) ,$ where $\eta \geq 1$ and $\psi :X\rightarrow
		\mathbb{R}_{+}$ is of the form $\psi \left( t\right) =at^{p}$ for some $p>1,$ where $%
		a\in
		\mathbb{R}_{+}.$ Then $q$ is a generalized $c$-distance on $X$ as shown in Example \ref{ex1.5}.
		
		Define a selfmap $T:X\rightarrow X$ by $T\left( x\right) =\lambda
		x $ for some $\lambda \in (0,1].$ Clearly, $T$ is a one to one, continuous
		and sequentially convergent map.
		
		Define $f:X\rightarrow X$ by%
		\begin{equation*}
			f(x)=\left \{
			\begin{array}{ll}
				\beta x^{2}\text{\ } & \text{if }x\neq 1\text{,} \\
				&  \\
				\beta x\  & \text{if }x=1,%
			\end{array}%
			\right.
		\end{equation*}%
		where $\beta \in \left( 0,\gamma \right) .$ A linear bounded operator $%
		A:E\rightarrow E$ is\ taken as $A=\left[
		\begin{array}{cc}
			a_{11} & a_{12} \\
			a_{21} & a_{22}%
		\end{array}%
		\right] $ with $a_{ij}\geq 0$ for $i,j\in \left \{ 1,2\right \} $. We take $%
		\beta ^{p}\leq a_{11}+\eta a_{12}$ and $\eta \beta ^{p}\leq a_{21}+\eta
		a_{22}.$
		
		Set $\left \Vert \mathbf{x}\right \Vert =\max \{ \left \vert x_{1}\right
		\vert ,\left \vert x_{2}\right \vert \}$ for $\mathbf{x}=\left[
		\begin{array}{c}
			x_{1} \\
			x_{2}%
		\end{array}%
		\right] ,$ $x_{i}\in
		\mathbb{R},$ $i=1,2,$ and
		$$\left \Vert A\mathbf{x}\right \Vert =\max \{ \left \vert
		a_{11}x_{1}+a_{12}x_{2}\right \vert ,\left \vert
		a_{21}x_{1}+a_{22}x_{2}\right \vert \}.$$
		Then,
		\begin{align*}
			\left \Vert A\right \Vert &=\sup \limits_{\left \Vert \mathbf{x}\right
				\Vert =1}\left \Vert A\mathbf{x}\right \Vert =\sup \limits_{\left \Vert
				\mathbf{x}\right \Vert =1}\left[ \max \{ \left \vert
			a_{11}x_{1}+a_{12}x_{2}\right \vert ,\left \vert
			a_{21}x_{1}+a_{22}x_{2}\right \vert \} \right] \\
			& =\max \{ \left \vert a_{11}+a_{12}\right \vert ,\left \vert a_{21}+a_{22}\right \vert \}.
		\end{align*}
		Now we are to show that
		\begin{equation*}
			q(Tf\left( x_{1}\right) ,Tf\left( x_{2}\right) )\leq A(U_{T}(x_{1},x_{2})),
		\end{equation*}%
		holds, where%
		\begin{equation*}
			U_{T}(x_{1},x_{2})\in \{q(T\left( x_{1}\right) ,T\left( x_{2}\right)
			),q(T\left( x_{1}\right) ,Tf\left( x_{1}\right) ),q(T\left( x_{2}\right)
			,Tf\left( x_{2}\right) )\}.
		\end{equation*}%
		Now for all $x,y\in X$ with $y\neq 1$, we obtain that%
		\begin{align*}
			q(Tf\left( x\right) ,Tf\left( y\right) )& =\left[
			\begin{array}{c}
				\psi \left( Tf\left( y\right) \right) \\
				\eta \psi \left( Tf\left( y\right) \right)%
			\end{array}%
			\right] ^{T}=\left[
			\begin{array}{c}
				a\left( Tf\left( y\right) \right) ^{p} \\
				\eta a\left( Tf\left( y\right) \right) ^{p}%
			\end{array}%
			\right] ^{T} \\
			& =\left[
			\begin{array}{c}
				a\left( \lambda \beta y^{2}\right) ^{p} \\
				\eta a\left( \lambda \beta y^{2}\right) ^{p}%
			\end{array}%
			\right] ^{T}\leq \left[
			\begin{array}{c}
				(a_{11}+\eta a_{12})\left[ a\left( \lambda y\right) ^{p}\right] \\
				(a_{21}+\eta a_{22})\left[ a\left( \lambda y\right) ^{p}\right]%
			\end{array}%
			\right] ^{T} \\
			& =\left[
			\begin{array}{cc}
				a_{11} & a_{12} \\
				a_{21} & a_{22}%
			\end{array}%
			\right] \left[
			\begin{array}{c}
				a\left( \lambda y\right) ^{p} \\
				\eta a\left( \lambda y\right) ^{p}%
			\end{array}%
			\right] ^{T} \\
			& =\left[
			\begin{array}{cc}
				a_{11} & a_{12} \\
				a_{21} & a_{22}%
			\end{array}%
			\right] \left[
			\begin{array}{c}
				\psi \left( T\left( y\right) \right) \\
				\eta \psi \left( T\left( y\right) \right)%
			\end{array}%
			\right] ^{T}=A(q(T\left( x\right) ,T\left( y\right) )),
		\end{align*}%
		where%
		\begin{equation*}
			q\left( T\left( x\right) ,T\left( y\right) \right) \in U_{T}\left(
			x,y\right) =\{q\left( T\left( x\right) ,T\left( y\right) \right) ,q(T\left(
			x\right) ,Tf\left( x\right) ),q(T\left( y\right) ,Tf\left( y\right) )\}.
		\end{equation*}%
		If $y=1$, then we have%
		\begin{align*}
			q(Tf\left( x\right) ,Tf\left( y\right) )& =\left[
			\begin{array}{c}
				\psi \left( Tf\left( y\right) \right) \\
				\eta \psi \left( Tf\left( y\right) \right)%
			\end{array}%
			\right] ^{T}=\left[
			\begin{array}{c}
				a\left( Tf\left( y\right) \right) ^{p} \\
				\eta a\left( Tf\left( y\right) \right) ^{p}%
			\end{array}%
			\right] ^{T} \\
			& =\left[
			\begin{array}{c}
				a\left( \lambda \beta \right) ^{p} \\
				\eta a\left( \lambda \beta \right) ^{p}%
			\end{array}%
			\right] ^{T}\leq \left[
			\begin{array}{c}
				(a_{11}+\eta a_{12})\left[ a\left( \lambda \right) ^{p}\right] \\
				(a_{21}+\eta a_{22})\left[ a\left( \lambda \right) ^{p}\right]%
			\end{array}%
			\right] ^{T} \\
			& =\left[
			\begin{array}{cc}
				a_{11} & a_{12} \\
				a_{21} & a_{22}%
			\end{array}%
			\right] \left[
			\begin{array}{c}
				a\left( \lambda \right) ^{p} \\
				\eta a\left( \lambda \right) ^{p}%
			\end{array}%
			\right] ^{T}=\left[
			\begin{array}{cc}
				a_{11} & a_{12} \\
				a_{21} & a_{22}%
			\end{array}%
			\right] \left[
			\begin{array}{c}
				\psi \left( T\left( y\right) \right) \\
				\eta \psi \left( T\left( y\right) \right)%
			\end{array}%
			\right] ^{T} \\
			& =A(q(T\left( x\right) ,T\left( y\right) )),
		\end{align*}%
		where%
		\begin{equation*}
			q\left( T\left( x\right) ,T\left( y\right) \right) \in U_{T}\left(
			x,y\right) =\{q\left( T\left( x\right) ,T\left( y\right) \right) ,q(T\left(
			x\right) ,Tf\left( x\right) ),q(T\left( y\right) ,Tf\left( y\right) )\}.
		\end{equation*}%
		So, the mapping $f:X\rightarrow X$ is a Perov $THR_{1}-$contraction. All
		conditions of Theorem\ref{2.1} are satisfied. Moreover, $u^{\ast }=0$ is a
		fixed point of $f$ and $q(u^{\ast },u^{\ast })=0$.
	\end{example}
	The following result is the direct consequence of Theorem \ref{2.1}.
	
	\begin{corollary}\label{2.33}
		Let $(X,d)$ be a complete $b$-cone
		metric space, $q$ a generalized $c$-distance on $X$ and $P$ a solid cone. If
		there exists\ a linear bounded operator $A:E\rightarrow E$ with $r(bA)<1$
		and $A\left( P\right) \subset P$ such that an inequalities%
		\begin{equation}\label{2.7}
			q(fx_{1},fx_{2})\leq A(U(x_{1},x_{2}))
		\end{equation}%
		hold for all $x_{1},x_{2}\in X$, where
		$
		U(x_{1},x_{2})\in \{q(x_{1},x_{2}),q(x_{1},fx_{1}),q(x_{2},fx_{2})\}.
		$
		Then $f$ has a unique fixed point $u^{\ast }\in X.$ Moreover, for any $%
		x_{0}\in X,$ iterative sequence $\{f^{n}x_{0}\}$ converges to the fixed
		point $u^{\ast }\in X$ and $q(u^{\ast },u^{\ast })=\theta $.
	\end{corollary}
	\begin{proof}
		Taking $T$ as an identity map in the
		Theorem \ref{2.1}, the result follows.
	\end{proof}
	\begin{example}
		\label{2.44}  \ Let $X=[0,1],$ $E=X\times X\ $and $P=\{
		\left( x,y\right) \in E:x,y\geq 0\}=E.$ Define $d:X\times X\rightarrow E$ by
		$d(x,y)=\left( |x-y|^{p},\alpha |x-y|^{p}\right) $, where $\alpha \geq 1$
		and $p>1.\ $Then $(X,d)$ is a cone $b$-metric space. Let $q:X\times
		X\rightarrow E$ be given by
		\begin{equation*}
			q(x,y)=\left( y^{p},\alpha y^{p}\right) .
		\end{equation*}%
		Then $q$ is a generalized $c$-distance. We fix $p=2.$ Define $f:X\rightarrow
		X$ by%
		\begin{equation*}
			f(x)=\left \{
			\begin{array}{ll}
				\dfrac{x^{2}}{10}\text{\ } & \text{if }x\neq 1\text{,} \\
				&  \\
				\dfrac{x}{10}\  & \text{if }x=1.%
			\end{array}%
			\right.
		\end{equation*}%
		Define a linear bounded operator $A:E\rightarrow E$ by $A=\left[
		\begin{array}{cc}
			\frac{2}{10} & \frac{1}{10} \\
			\frac{3}{10} & \frac{1}{10}%
		\end{array}%
		\right] $. Set $\left \Vert \mathbf{x}\right \Vert =\max \{ \left \vert
		x_{1}\right \vert ,\left \vert x_{2}\right \vert \}$ for $\mathbf{x}=\left[
		\begin{array}{c}
			x_{1} \\
			x_{2}%
		\end{array}%
		\right] ,$ $x_{i}\in
		\mathbb{R}
		,$ $i=1,2.$ For an arbitrary $\mathbf{x}\in E,$ we have%
		\begin{eqnarray*}
			\left \Vert A\mathbf{x}\right \Vert &=&\max \left \{ \left \vert \frac{2}{10}%
			x_{1}+\frac{1}{10}x_{2}\right \vert ,\left \vert \frac{3}{10}x_{1}+\frac{1}{%
				10}x_{2}\right \vert \right \} \\
			&\leq &\max \left \{ \frac{3}{10}\max \left \{ \left \vert x_{1}\right \vert
			,\left \vert x_{2}\right \vert \right \} ,\frac{4}{10}\left \vert
			x_{2}\right \vert \right \} =\frac{3}{10}\left \Vert \mathbf{x}\right \Vert .
		\end{eqnarray*}%
		Thus, $\left \Vert A\right \Vert \leq \frac{3}{10}.$ If $\mathbf{x}=\left[
		\begin{array}{c}
			1 \\
			1%
		\end{array}%
		\right] ,$ then $\left \Vert \mathbf{x}\right \Vert =1$ and
		$$\left \Vert A \mathbf{x}\right \Vert =\max \{ \left \vert \frac{2}{10}+\frac{1}{10}%
		\right \vert ,\left \vert \frac{3}{10}+\frac{1}{10}\right \vert \}=\max \{
		\frac{3}{10},\frac{4}{10}\}=\frac{3}{10}.$$ Hence $\left \Vert A\right \Vert =%
		\frac{3}{10}=0.3.$ Clearly $\left \Vert bA\right \Vert =\left \Vert 2A\right
		\Vert =0.6<1$ and $A\left( P\right) \subset P.$
		
		Now for all $x,y\in X$ with $y\neq 1$, we obtain that
		\begin{align*}
			q(fx,fy)& =\left[
			\begin{array}{c}
				\frac{y^{4}}{100} \\
				\frac{\alpha y^{4}}{100}
			\end{array}
			\right] ^{T}
			\leq \left[
			\begin{array}{c}
				(\frac{2}{10}+\frac{\alpha }{10})y^{2} \\
				(\frac{3}{10}+\frac{\alpha }{10})y^{2}%
			\end{array}%
			\right] ^{T} \\
			& =\left[
			\begin{array}{cc}
				\frac{2}{10} & \frac{1}{10} \\
				\frac{3}{10} & \frac{1}{10}%
			\end{array}%
			\right] \left[
			\begin{array}{c}
				y^{2} \\
				\alpha y^{2}%
			\end{array}%
			\right] ^{T}=A(q(x,y)),
		\end{align*}%
		where
		$
		q\left( x,y\right) \in U\left( x,y\right) =\{q\left( x,y\right)
		,q(x,fx),q(y,fy)\}.
		$
		If $y=1$, then
		\begin{align*}
			q(fx,fy)& =\left[
			\begin{array}{c}
				\frac{y^{2}}{100} \\
				\frac{\alpha y^{2}}{100}
			\end{array}
			\right] ^{T}  \leq \left[
			\begin{array}{c}
				(\frac{2}{10}+\frac{\alpha }{10})y^{2} \\
				(\frac{3}{10}+\frac{\alpha }{10})y^{2}
			\end{array}
			\right] ^{T} \\
			& =\left[
			\begin{array}{cc}
				\frac{2}{10} & \frac{1}{10} \\
				\frac{3}{10} & \frac{1}{10}
			\end{array}
			\right] \left[
			\begin{array}{c}
				y^{2} \\
				\alpha y^{2}
			\end{array}
			\right] ^{T}=A(q(x,y)),
		\end{align*}
		where
		$$q\left( x,y\right) \in U\left( x,y\right) =\{q\left( x,y\right),q(x,fx),q(y,fy)\}.$$
		Hence, all the conditions of Corollary \ref{2.33} are satisfied. Moreover, $u^{\ast
		}=0$ is a fixed point of $f$ and $q(u^{\ast },u^{\ast })=0$.
	\end{example}
	
	The following result is the direct consequence of Theorem \ref{2.1}
	which extends and generalizes \cite[Theorem 1]{FA13}.
	\begin{corollary}
		\label{c2.5} Let $(X,d)$ be a complete cone metric
		space, $q$ a generalized $c$-distance on $X$ and $P$ a solid cone. Suppose
		that $T:X\rightarrow X$ is a one to one, continuous function and
		sequentially convergent and $f:X\rightarrow X$ . If there exists\ a linear
		bounded operator $A:E\rightarrow E$ with $r(bA)<1$ and $A\left( P\right)
		\subset P$ such that%
		\begin{equation}\label{e2.8}
			q(Tfx_{1},Tfx_{2})\leq A(q(x_{1},x_{2})),
		\end{equation}%
		hold for all $x_{1},x_{2}\in X$, then $f$ has a unique fixed point $u^{\ast
		}\in X.$ Moreover, for any $x_{0}\in X,$ the iterative sequence $%
		\{f^{n}x_{0}\}$ converges to the fixed point $u^{\ast }\in X$ and $%
		q(Tu^{\ast },Tu^{\ast })=\theta $.
	\end{corollary}
	\begin{theorem}
		\label{t2.6} Let $(X,d)$ be a complete cone $b$%
		-metric space, $q$ a generalized $c$-distance on $X\ $and $P$ a solid cone.
		Suppose that $T:X\rightarrow X$ is a one to one, continuous function and
		sequentially convergent. If the mapping $f:X\rightarrow X$ is a Perov $%
		THR_{2}-$contraction, then $f$ has a unique fixed point $u^{\ast }\in X.$
		Also, for any $x_{0}\in X,$ the iterative sequence $\{f^{n}x_{0}\}$
		converges to the fixed point of $f$. Moreover $q(Tu^{\ast },Tu^{\ast
		})=\theta $ provided that $u^{\ast }=fu^{\ast }.$
	\end{theorem}
	\begin{proof}
		Let $x_{0}$ be any point in $X.$ Define a
		sequence $\{x_{n}\}$ in $X$ by
		$x_{1}=fx_{0},$ $x_{2}=fx_{1}=f^{2}x_{0},$ $...,$ $x_{n+1}=fx_{n-1}=f^{n}x_{0}.$ Then for $n\geq 0,$ we have%
		\begin{eqnarray*}
			q(Tx_{n},Tx_{n+1}) &=&q(Tfx_{n-1},Tfx_{n}) \\
			&\leq &A_{1}(q(Tx_{n-1},Tx_{n}))+A_{2}(q(Tx_{n-1},Tfx_{n-1})) \\
			&&+A_{3}(q(Tx_{n},Tfx_{n})) \\
			&=&A_{1}(q(Tx_{n-1},Tx_{n}))+A_{2}(q(Tx_{n-1},Tx_{n})) \\
			&&+A_{3}(q(Tx_{n},Tx_{n+1})) \\
			&=&(A_{1}+A_{2})(q(Tx_{n-1},Tx_{n}))+A_{3}(q(Tx_{n},Tx_{n+1})),
		\end{eqnarray*}
		which implies that
		\begin{equation*}
			q(Tx_{n},Tx_{n+1})\leq (A_{1}+A_{2})(I-A_{3})^{-1}(q(Tx_{n-1},Tx_{n})).
		\end{equation*}
		Thus, for all $n\in \mathbb{N\cup \{}0\},$
		$$q(Tx_{n},Tx_{n+1})\leq A(q(Tx_{n-1},Tx_{n})),$$
		where $A=(A_{1}+A_{2})(I-A_{3})^{-1}$
		with $(b||A_{1}||+||A_{2}||)(I-||A_{3}||)^{-1}<1.$
		Continuing this way, we have for $n\in \mathbb{N\cup \{}0\},$
		\begin{equation*}
			q(Tx_{n},Tx_{n+1})\leq A^{n}(q(Tx_{0},Tx_{1})).
		\end{equation*}
		For $m,n\in
		\mathbb{N}
		\mathbb{\cup \{}0\}$ with $m>n,$ it follows from \ref{q2:defgencdist} that
		\begin{align*}
			q(Tx_{n},Tx_{m})& \leq b[q(Tx_{n},Tx_{n+1})+q(Tx_{n+1},Tx_{m})] \\
			& \leq bq(Tx_{n},Tx_{n+1})+b[bq(Tx_{n+1},Tx_{n+2})+q(Tx_{n+2},Tx_{m})] \\
			& \leq \cdot \cdot \cdot \\
			& \leq bq(Tx_{n},Tx_{n+1})+b^{2}q(Tx_{n+1},Tx_{n+2})+\cdots
			+b^{m-n}q(Tx_{m-1},Tx_{m}) \\
			& \leq (bA^{n}+b^{2}A^{n+1}\cdots +b^{m-n}A^{m-1})(q(Tx_{0},Tx_{1})) \\
			& \leq bA^{n}\left( I-bA\right) ^{-1}(q(Tx_{0},Tx_{1})).
		\end{align*}
		Let $c\gg 0.$ Choose $\delta >0$ such that $c+N_{\delta }(\theta )\subseteq
		P\,,$ where $N_{\delta }(\theta )=\{x\in E:\left \Vert x\right \Vert <\delta
		\}.$ Also, choose $N_{1}\in\mathbb{N}$ such that $bA^{n}(I-bA)^{-1}(q(Tx_{0},Tx_{1}))\in N_{\delta }(\theta )\ $
		for all $n>N_{1}$. Thus for all $m>n>N_{1},$ we have
		\begin{equation*}
			q(Tx_{n},Tx_{m})\preceq bA^{n}(I-bA)^{-1}(d(Tx_{0},Tx_{1}))\ll c
		\end{equation*}
		which implies that $\{Tx_{n}\}$ is a Cauchy sequence in $X$. Since $X$ is
		complete, there exists an element $v\in X$ such that $Tx_{n}\rightarrow v$
		as $n\rightarrow \infty $.
		
		Since $T$ is subsequently convergent, $\{x_{n}\}$has a convergent
		subsequence. So, there are $x^{\ast }\in X$ and a subsequence $\{x_{n_{i}}\}$
		such that $x_{n_{i}}\rightarrow x^{\ast }$ as $i\rightarrow \infty $. As $T$
		is continuous, we obtain that $\lim \limits_{n_{i}\rightarrow \infty}Tx_{n_{i}}=Tx^{\ast }$. The uniqueness of the limit implies that
		$Tx^{\ast}=v$. By \ref{q3:defgencdist}, we get that
		\begin{equation}\label{e2.9}
			q(Tx_{n},Tx^{\ast })\leq bA^{n}(I-bA)^{-1}(q(Tx_{0},Tx_{1})),
		\end{equation}
		where $A=A_{1}+A_{2}+A_{3}$ with $b\left( \left \Vert A_{1}\right \Vert
		+\left \Vert A_{2}\right \Vert +\left \Vert A_{3}\right \Vert \right) <1.$
		Also,
		\begin{equation}\label{e2.10}
			\begin{array}{l}
				q(Tx_{n},Tfx^{\ast }) =q(Tfx_{n-1},Tfx^{\ast }) \\
				\quad \leq A_{1}\left( q(Tx_{n-1},Tx^{\ast })\right) +A_{2}\left(
				q(Tx_{n-1},Tfx_{n-1})\right) +A_{3}\left( q(Tx^{\ast },Tfx^{\ast })\right)  \\
				\quad =A_{1}\left( q(Tx_{n-1},Tx^{\ast })\right) +A_{2}\left(
				q(Tx_{n-1},Tx_{n})\right) +A_{3}\left( q(Tx^{\ast },Tfx^{\ast })\right)  \\
				\quad \leq A^{n-1}(I-A)^{-1}(q(Tx_{0},Tx_{1}))+A^{n}(I-A)^{-1}(q(Tx_{0},Tx_{1}))
				\\
				\quad \quad \quad \quad \quad \quad \quad +\left( I-A\right) ^{-1}A^{n}\left( q(Tx_{0},Tx_{1})\right),
			\end{array}
		\end{equation}
		where $A=A_{1}+A_{2}+A_{3}$ with $b\left( \left \Vert A_{1}\right \Vert
		+\left \Vert A_{2}\right \Vert +\left \Vert A_{3}\right \Vert \right) <1.$
		By Lemma \ref{lemma:cdistconv} \ref{ilemma:cdistconv}, and \eqref{e2.9} and \eqref{e2.10}, we get that $Tx^{\ast}=Tfx^{\ast }$. As the map $T$ is one to one, we obtain that $x^{\ast}=fx^{\ast }.$
		Now by taking $x^{\ast }=fx^{\ast },$ we have
		\begin{eqnarray*}
			q(Tx^{\ast },Tx^{\ast }) &=&q(Tfx^{\ast },Tfx^{\ast }) \\
			&\leq &A_{1}\left( q(Tx^{\ast },Tx^{\ast })\right) +A_{2}\left( q(Tx^{\ast
			},Tfx^{\ast })\right) +A_{3}\left( q(Tx^{\ast },Tfx^{\ast })\right)  \\
			&=&A_{1}\left( q(Tx^{\ast },Tx^{\ast })\right) +A_{2}\left( q(Tx^{\ast
			},Tx^{\ast })\right) +A_{3}\left( q(Tx^{\ast },Tx^{\ast })\right)  \\
			&=&\left( A_{1}+A_{2}+A_{3}\right) \left( q(Tx^{\ast },Tx^{\ast })\right)  \\
			&=&A(q(Tx^{\ast },Tx^{\ast })),
		\end{eqnarray*}
		where $A=A_{1}+A_{2}+A_{3}$ with $b\left( \left \Vert A_{1}\right \Vert
		+\left \Vert A_{2}\right \Vert +\left \Vert A_{3}\right \Vert \right) <1.$ By
		Remark \ref{1.10}, we obtain that $q(Tx^{\ast },Tx^{\ast })=\theta .$
		
		To prove the uniqueness of fixed point, suppose that there exists
		another point $y^{\ast }$ in $X$ such that $y^{\ast }=fy^{\ast }$. Then,
		\begin{eqnarray*}
			q(Tx^{\ast },Ty^{\ast }) &=&q(Tfx^{\ast },Tfy^{\ast }) \\
			&\leq &A_{1}(q(Tx^{\ast },Ty^{\ast }))+A_{2}(q(Tx^{\ast },Tfx^{\ast
			}))+A_{3}(q(Ty^{\ast },Tfy^{\ast })) \\
			&=&A_{1}(q(Tx^{\ast },Ty^{\ast }))+A_{2}(q(Tx^{\ast },Tx^{\ast
			}))+A_{3}(q(Ty^{\ast },Ty^{\ast })) \\
			&=&A_{1}(q(Tx^{\ast },Ty^{\ast }))+A_{2}(\theta )+A_{3}(\theta ) \\
			&=&A_{1}(q(Tx^{\ast },Ty^{\ast })).
		\end{eqnarray*}
		As $\left \Vert A_{1}\right \Vert <1$ and $A_{1}\left( P\right) \subset
		A_{1},\,$ by the Remark \ref{1.10}, we obtain $q(Tx^{\ast },Ty^{\ast })=\theta $.
		Also, by Lemma \ref{lemma:cdistconv}\ref{ilemma:cdistconv}, we get $Tx^{\ast }=Ty^{\ast }.$
		
		Since $T$ is one to one, so $x^{\ast }=y^{\ast }$. Thus $f$ has a
		unique fixed point.
	\end{proof}
	
	The following result is the direct consequence of Theorem \ref{t2.6}
	which extends and generalizes \cite[Corollary 1]{FA13}.
	
	\begin{corollary}
		\label{c2.7}
		Let $(X,d)$ be a complete cone
		$b$-metric space, $q$ a $c$-distance on $X$ and $P$ a solid cone. If there
		exists\ linear bounded operators $A_{1},A_{2},A_{3}:E\rightarrow E$ with $%
		A_{k}\left( P\right) \subset P$ for $k=1,2,3$, $(I-A_{3})^{-1}\in \mathcal{B}%
		(E)$ and $r[b(I-A_{3})^{-1}(A_{1}+A_{2})]<1$ such that%
		\begin{equation}\label{e2.11}
			q(fx_{1},fx_{2})\leq A_{1}\left( q(x_{1},x_{2})\right) +A_{2}\left(
			q(x_{1},fx_{1})\right) +A_{3}\left( q(x_{2},fx_{2})\right)
		\end{equation}
		holds for all $x_{1},x_{2}\in X$. Then $f$ has a unique fixed point $u^{\ast
		}\in X.$ Also, for any $x_{0}\in X,$ iterative sequence $\{f^{n}x_{0}\}$
		converges to the fixed point $u^{\ast }\in X$. Moreover $q(u^{\ast },u^{\ast
		})=\theta $.
	\end{corollary}
	\begin{proof}
		Taking $T$ as an identity map in Theorem \ref{t2.6},
		the result follows.
	\end{proof}
	\begin{example}
		\label{ex2.8} Let $E=\mathbb{R}^{2},$ $X=[0,1]$ and $%
		P=\{(x,y)\in E:x,y\geq 0\}.$ Let the the cone $b$-metric $d:X\times
		X\rightarrow E$ be given by $d(x,y)=(|x-y|^{p},\alpha |x-y|^{p})$ where
		$\alpha \geq 1$. Define $q:X\times X\rightarrow E$ by
		$q(x,y)=(y^{p},\alpha y^{p}).$
		Define $f:X\rightarrow X$ by%
		\begin{equation*}
			f(x)=\left \{
			\begin{array}{cc}
				\frac{x^{2}}{4} & \text{if }x\neq 1, \\
				&  \\
				\frac{1}{4} & \text{if }x=1.%
			\end{array}%
			\right.
		\end{equation*}%
		Define the bounded linear operators $A_{1},A_{2},A_{3}:E%
		\rightarrow E$ by $A_{1}=\left[
		\begin{array}{cc}
			\frac{1}{5} & \frac{1}{6} \\
			0 & \frac{3}{13}%
		\end{array}%
		\right] ,$ $A_{2}=\left[
		\begin{array}{cc}
			\frac{1}{7} & 0 \\
			0 & \frac{1}{8}%
		\end{array}%
		\right] ,$ $A_{3}=\left[
		\begin{array}{cc}
			\frac{1}{6} & 0 \\
			0 & \frac{2}{9}%
		\end{array}%
		\right] $. Set $\left \Vert \mathbf{x}\right \Vert =\max \{ \left \vert
		x_{1}\right \vert ,\left \vert x_{2}\right \vert \}$ for $\mathbf{x}=\left[
		\begin{array}{c}
			x_{1} \\
			x_{2}%
		\end{array}%
		\right] ,$ $x_{i}\in
		\mathbb{R}
		,$ $i=1,2.$ For an arbitrary $\mathbf{x}\in E,$ we have%
		\begin{eqnarray*}
			\left \Vert A_{1}\mathbf{x}\right \Vert  &=&\max \left \{ \left \vert \frac{1}{5}%
			x_{1}+\frac{1}{6}x_{2}\right \vert ,\left \vert 0x_{1}+\frac{3}{13}%
			x_{2}\right \vert \right \}  \\
			&\leq &\max \left \{ \frac{11}{30}\max \left \{ \left \vert x_{1}\right \vert
			,\left \vert x_{2}\right \vert \right \} ,\frac{3}{13}\left \vert
			x_{2}\right \vert \right \} =\frac{11}{13}\left \Vert \mathbf{x}\right \Vert .
		\end{eqnarray*}%
		Thus, $\left \Vert A_{1}\right \Vert \leq \frac{11}{13}.$ If $\mathbf{x}=\left[
		\begin{array}{c}
			1 \\
			1%
		\end{array}%
		\right] ,$ then $\left \Vert \mathbf{x}\right \Vert =1,$ and
		$$\left \Vert A_{1}\mathbf{x}\right \Vert =\max \{ \left \vert \frac{1}{5}+\frac{1}{6}\right \vert
		,\left \vert 0+\frac{3}{13}\right \vert \}=\max \{ \frac{11}{30},\frac{3}{13}\}=
		\frac{11}{30}.$$
		Hence $\left \Vert A_{1}\right \Vert =\frac{11}{30}.$
		Similarly, $\left \Vert A_{2}\right \Vert =\frac{1}{7}$ and $
		\left
		\Vert A_{3}\right \Vert =\frac{2}{9}.$ Thus,
		$$\left \Vert
		A_{1}\right
		\Vert +\left \Vert A_{2}\right \Vert +\left \Vert
		A_{3}\right
		\Vert =0.7317<1.$$
		Then, for all $x,y\in X$ and $y\neq 1,$ we have
		\begin{eqnarray*}
			q(fx,fy) &=&q(\frac{x^{2}}{4},\frac{y^{2}}{4})
			=\left[
			\begin{array}{c}
				\frac{y^{2p}}{16} \\
				\frac{\alpha y^{2p}}{16}
			\end{array}
			\right] ^{T}
			\leq \left[
			\begin{array}{c}
				\frac{1}{5}y^{2}+\frac{1}{6}\alpha y^{2}+\frac{1}{112}x^{4}+\frac{1}{96}y^{4}
				\\
				\frac{3}{13}\alpha y^{2}+\frac{1}{128}\alpha x^{4}+\frac{1}{72}\alpha y^{4}
			\end{array}
			\right] ^{T} \\
			&=&\left[
			\begin{array}{cc}
				\frac{1}{5} & \frac{1}{6} \\
				0 & \frac{3}{13}%
			\end{array}%
			\right] \left[
			\begin{array}{c}
				y^{2} \\
				\alpha y^{2}%
			\end{array}%
			\right] +\left[
			\begin{array}{cc}
				\frac{1}{7} & 0 \\
				0 & \frac{1}{8}%
			\end{array}%
			\right] \left[
			\begin{array}{c}
				\frac{x^{4}}{16} \\
				\frac{\alpha x^{4}}{16}%
			\end{array}%
			\right] +\left[
			\begin{array}{cc}
				\frac{1}{6} & 0 \\
				0 & \frac{2}{9}%
			\end{array}%
			\right] \left[
			\begin{array}{c}
				\frac{y^{4}}{16} \\
				\frac{\alpha y^{4}}{16}%
			\end{array}%
			\right]  \\
			&\leq &A_{1}(q(x,y))+A_{2}(q(x,fx))+A_{3}(q(y,fy)).
		\end{eqnarray*}%
		Also for $x,y\in X$ with $y=1,$ we obtain
		\begin{eqnarray*}
			q(fx,fy) &=&q(\frac{1}{4},\frac{1}{4})
			=\left[
			\begin{array}{c}
				\frac{1}{16} \\
				\frac{\alpha }{16}
			\end{array}
			\right]^{T}
			\leq \left[
			\begin{array}{c}
				\frac{1}{5}+\frac{1}{6}\alpha +\frac{1}{112}x^{4}+\frac{1}{96} \\
				\frac{3\alpha }{13}+\frac{\alpha }{128}x^{4}+\frac{1}{72}\alpha
			\end{array}%
			\right] ^{T} \\
			&=&\left[
			\begin{array}{cc}
				\frac{1}{5} & \frac{1}{6} \\
				0 & \frac{3}{13}%
			\end{array}%
			\right] \left[
			\begin{array}{c}
				1 \\
				\alpha
			\end{array}%
			\right] +\left[
			\begin{array}{cc}
				\frac{1}{7} & 0 \\
				0 & \frac{1}{8}%
			\end{array}%
			\right] \left[
			\begin{array}{c}
				\frac{x^{4}}{16} \\
				\frac{\alpha x^{4}}{16}%
			\end{array}%
			\right] +\left[
			\begin{array}{cc}
				\frac{1}{6} & 0 \\
				0 & \frac{2}{9}%
			\end{array}%
			\right] \left[
			\begin{array}{c}
				\frac{1}{16} \\
				\frac{\alpha }{16}%
			\end{array}%
			\right]  \\
			&\leq &A_{1}(q(x,y))+A_{2}(q(x,fx))+A_{3}(q(y,fy)).
		\end{eqnarray*}%
		Thus mapping $f$ satisfies the condition \eqref{e2.11}. Hence all the
		conditions of Corollary \ref{c2.7} are satisfied. Moreover, $u^{\ast }=0$ is a
		fixed point of $f$, and $q(u^{\ast },u^{\ast })=0.$
	\end{example}

\end{document}